\documentclass[12pt,a4paper]{amsart}

\newtheorem{theo+}              {Theorem}           [section]
\newtheorem{prop+}  [theo+]     {Proposition}
\newtheorem{coro+}  [theo+]     {Corollary}
\newtheorem{lemm+}  [theo+]     {Lemma}
\newtheorem{exam+}  [theo+]     {Example}
\newtheorem{rema+}  [theo+]     {Remark}
\newtheorem{defi+}  [theo+]     {Definition}
\def \r{\mbox{${\mathbb R}$}}

\theoremstyle{plain}
\newtheorem{thm}{Theorem}
\newtheorem{prop}{Proposition}
\newtheorem{cor}{Corollary}

\theoremstyle{definition}

\newtheorem{rk}{Remark}

\usepackage{amsthm}
\theoremstyle{plain} \theoremstyle{remark}
\newtheorem{remark}{Remark}
\newtheorem{example}{Example}

\def\E{/\kern-1.0em \equiv }

\title{On  conformal biharmonic immersions}
\author{Ye-Lin Ou$^{*}$ }

\address{Department of Mathematics,\newline\indent
Texas A $\&$ M University-Commerce,\newline\indent Commerce, TX
75429, U S A.\newline\indent E-mail:yelin$\_$ou@tamu-commerce.edu}

\thanks{$^*$ Supported by Texas A $\&$ M University-Commerce
Faculty Development Program (2008)}
\begin{document}

\title[On Conformal biharmonic immersions]{On Conformal biharmonic immersions}

\thanks{}
\date{08/18/08}

\maketitle
\section*{Abstract}
\begin{quote}
{\footnotesize  This paper studies conformal biharmonic immersions.
We first study the transformations of Jacobi operator and the
bitension field under conformal change of metrics. We then obtain an
invariant equation for a conformal biharmonic immersion of a surface
into Euclidean $3$-space. As applications, we construct a
$2$-parameter family of non-minimal conformal biharmonic immersions
of cylinder into $\r^3$ and some examples of conformal biharmonic
immersions of $4$-dimensional Euclidean space into sphere and
hyperbolic space thus provide many simple examples of proper
biharmonic maps with rich geometric meanings. These suggest that
there are abundant proper biharmonic maps in the family of conformal
immersions. We also explore the relationship between biharmonicity
and holomorphicity of conformal immersions of surfaces.}
\end{quote}

\section{Introduction}
This paper works on the smooth objects, so we assume that
manifolds, maps, vector fields, etc, are smooth unless it is stated otherwise.\\

A biharmonic map is a map $\varphi:(M, g)\longrightarrow (N, h)$
between Riemannian manifolds that is a critical point of the
bienergy functional
\begin{equation}\nonumber
E^{2}\left(\varphi,\Omega \right)= \frac{1}{2} {\int}_{\Omega}
\left|\tau(\varphi) \right|^{2}{\rm d}x
\end{equation}
for every compact subset $\Omega$ of $M$, where $\tau(\varphi)={\rm
Trace}_{g}\nabla {\rm d} \varphi$ is the tension field of $\varphi$.
The Euler-Lagrange equation of this functional gives the biharmonic
map equation (\cite{Ji1})
\begin{equation}\label{BI1}
\tau^{2}(\varphi):={\rm
Trace}_{g}(\nabla^{\varphi}\nabla^{\varphi}-\nabla^{\varphi}_{\nabla^{M}})\tau(\varphi)
- {\rm Trace}_{g} R^{N}({\rm d}\varphi, \tau(\varphi)){\rm d}\varphi
=0,
\end{equation}
which states the fact that the map $\varphi$ is biharmonic if and
only if its bitension field $\tau^{2}(\varphi)$ vanishes
identically. In the above equation we have used $R^{N}$ to denote
the curvature operator of $(N, h)$ defined by
$$R^{N}(X,Y)Z=
[\nabla^{N}_{X},\nabla^{N}_{Y}]Z-\nabla^{N}_{[X,Y]}Z.$$

\indent Harmonic maps are clearly biharmonic, so it is more
interesting to study {\em proper} (meaning non-harmonic) biharmonic
maps as far as one seeks to pursuit a new study. However, apart from
the maps between Euclidean spaces defined by polynomials of degree
less than four (a class of maps that seems so wild to exhibit any
characteristic property) not many examples of proper biharmonic maps
between Riemannain manifolds have been found (see, e.g., \cite{MO},
\cite{LO2}, \cite{Ou1}, and the bibliography of biharmonic maps
\cite{LMO}). So, currently, one priority and a practical thing to do
seems to be finding more examples of proper biharmonic maps between
certain model spaces or studying biharmonic maps under some
geometric constraints. For example, one can study biharmonic
isometric immersions which lead to the concept of biharmonic
submanifolds (see e.g., \cite{Ji2}, \cite{CH}, \cite{CI},
 \cite{CMO1}, \cite{CMO2}, \cite{MO} and \cite{BMO}); one can also study, as in
 \cite{BK}, \cite{BFO},  \cite{LO2}, horizontally weakly conformal biharmonic maps which
generalize both the notion of harmonic morphisms (maps that are both
horizontally weakly conformal and harmonic) and that of biharmonic
morphisms (maps that are horizontally weakly conformal biharmonic
with other constraints, see \cite{OU1}, \cite{LO1}, \cite{Ou1}, and \cite{LO2} for details).\\

The interesting link between harmonicity and conformality has a long
history. It was known to Weierstrass that a conformal immersion
$\varphi: M^2\longrightarrow \r^{3}$ is harmonic if and only if
$\varphi(M)$ is a minimal submanifold of $\r^3$. It is also well
known that {\bf conformal harmonic immersions} of surfaces are
precisely {\bf conformal minimal immersions} of surfaces of which
there has been a rich theory exhibiting a beautiful interplay among
geometry, topology, and real and complex analysis.  So it would be
interesting to know if we can generalize (or use the tools of) the
theory on conformal minimal immersions to conformal biharmonic
immersions. On the other hand, Jiang and Chen-Ishikawa independently
proved that an isometric immersion $\varphi: M^2\longrightarrow
\r^{3}$ is biharmonic if and only if $\varphi$ is harmonic. It would
also be interesting to know whether this result can be generalized
to the case of conformal biharmonic immersions. Motivated by these,
we study conformal biharmonic immersions in this paper. First, we
study the transformations of Jacobi operator and bitension field
under conformal change of metrics. We then obtain an invariant
equation for a conformal biharmonic immersion of a surface into
Euclidean $3$-space, and using this, we construct a $2$-parameter
family of non-minimal conformal biharmonic immersions of cylinder
into $\r^3$ and some examples of conformal immersions of
$4$-dimensional Euclidean space into sphere and hyperbolic space,
thus provide many simple examples of proper biharmonic maps with
rich geometric meanings. We also explore Weierstrass type
representations for conformal biharmonic immersions.

\section{Jacobi operators and the bitension fields under conformal change of metrics}

For a map $\varphi : (M^{m},g) \longrightarrow (N^{n},h)$, the
Jacobi operator is defined as
\begin{equation}\label{JO}
J^{\varphi}_{g}(X)=-\{{\rm
Trace}_{g}(\nabla^{\varphi}\nabla^{\varphi}-\nabla^{\varphi}_{\nabla^{M}})X
- {\rm Trace}_{g} R^{N}({\rm d}\varphi, X){\rm d}\varphi\}
\end{equation}
for any vector field $X$ along the map $\varphi$. Thus, by
(\ref{BI1}) and (\ref{JO}), the relationship between the Jacobi
operator and the bitension field of $\varphi$ is explained by
$J^{\varphi}_{g}(\tau(\varphi))=-\tau^2(\varphi)$.
\begin{thm}\label{confC}
Let $\varphi : (M^{m},g) \longrightarrow (N^{n},h)$ be a map. Then,
under the conformal change of metrics ${\bar g}=F^{-2}g$, we have
\begin{itemize}
\item[(I)] the transformation of the Jacobi operators $J^{\varphi}_{g}$ and $J^{\varphi}_{\bar g}$ of
$\varphi$ is given by
\begin{equation}\label{JF}
J^{\varphi}_{\bar g}(X)= F^2J^{ \varphi}_{g}(X)+F^2(m-2)\nabla^{
\varphi}_{{\rm grad\,ln}F}X,
\end{equation}
and
\item[(II)] the transformation of the bitension fields $\tau^{2}(\varphi, g)$ and  $\tau^{2}(\varphi,{\bar g})$ of
$\varphi$ is given by
\begin{eqnarray}\label{tao2}
&&\tau^{2}(\varphi,{\bar g})= F^4\{\tau^{2}(\varphi, g)+(m-2)J^{
\varphi}_{g}({\rm d}{\varphi}({\rm grad\,ln}F))\\\notag && +
2(\Delta {\rm ln}F-(m-4)\left|{\rm grad\,ln}F\right|^2)\tau(\varphi,
g)-(m-6)\nabla^{ \varphi}_{{\rm grad\,ln}\,F}\,\tau(\varphi,
g)\\\notag && -2(m-2)(\Delta {\rm ln}F-(m-4)\left|{\rm
grad\,ln}F\right|^2){\rm d}{\varphi}({\rm grad\,ln}F)\\\notag
&&+(m-2)(m-6)\nabla^{ \varphi}_{{\rm grad\,ln}\,F}\,{\rm
d}{\varphi}({\rm grad\,ln}F)\},
\end{eqnarray}
where ${\rm grad}$ and $\Delta$ denote the gradient and the
Laplacian taken with respect to the metric $g$.
\end{itemize}
\end{thm}

\begin{proof}

Choose a local orthonormal frames $\{e_i\}$ with respect to $g$ on
$M$, then $\{{\bar e}_i=Fe_i\}$ is a local orthonormal frames with
respect to ${\bar g}$.\\

A direct computation gives the transformation of the tension fields
under the conformal change of a metric as
\begin{eqnarray}\notag
 \tau(\varphi,{\bar g})&=& F^2\{\tau (\varphi, g)-(m-2){\rm
d}{\varphi}({\rm grad\,ln}F)\}.
\end{eqnarray}

Also, a straightforward computation (see, e.g., \cite{BK}) yields
\begin{eqnarray}\label{J2}
&&{\rm Trace}_{\bar
g}(\nabla^{\varphi}\nabla^{\varphi}-\nabla^{\varphi}_{{\bar\nabla}^{M}})X\\\notag
&=&F^2\{{\rm
Trace}_{g}(\nabla^{\varphi}\nabla^{\varphi}-\nabla^{\varphi}_{\nabla^{M}})X-(m-2)\nabla^{\varphi}_{{\rm
grad\,ln}F}X\}.
\end{eqnarray}

On the other hand,
\begin{eqnarray}\label{J3}
{\rm Trace}_{\bar g} R^{N}({\rm d}\varphi, X){\rm d}\varphi & =&
\sum_{i=1}^{m}R^{N}({\rm d}\varphi (Fe_i), X){\rm
d}\varphi(Fe_i)\\\notag & = & F^2{\rm Trace}_{g} R^{N}({\rm
d}\varphi, X){\rm d}\varphi.
\end{eqnarray}
Using Equations (\ref{J2}) and (\ref{J3}) we have
\begin{eqnarray}\notag
&&-J^{\varphi}_{\bar g}(X)={\rm Trace}_{\bar
g}(\nabla^{\varphi}\nabla^{\varphi}-\nabla^{\varphi}_{{\bar\nabla}^{M}})X-
{\rm Trace}_{\bar g} R^{N}({\rm d}\varphi, X){\rm d}\varphi\\\notag
&=& F^2\{{\rm
Trace}_{g}(\nabla^{\varphi}\nabla^{\varphi}-\nabla^{\varphi}_{\nabla^{M}})X-
{\rm Trace}_{g} R^{N}({\rm d}\varphi, X){\rm d}\varphi\}\\\notag
&&-(m-2)F^2\nabla^{\varphi}_{{\rm grad\,ln}F}X\\\notag &&= - F^2J^{
\varphi}_{g}(X)-(m-2)F^2\nabla^{ \varphi}_{{\rm grad\,ln}F}X,
\end{eqnarray}
from which we obtain part (I) of the theorem.\\

To prove the second part of the Theorem, we compute
\begin{eqnarray}\notag\label{J100}
J^{\varphi}(f X)&=&-\{{\rm
Trace}_{g}(\nabla^{\varphi}\nabla^{\varphi}-\nabla^{\varphi}_{\nabla^{M}})(f
X)- {\rm Trace}_{g} R^{N}({\rm d}\varphi, fX){\rm
d}\varphi\}\\\label{J100} &=& fJ^{\varphi}(X)-(\Delta f)X-2
\nabla^{\varphi}_{{\rm grad}f}X, \\\label{LO0} &&\Delta
F^2=2F^2\Delta {\rm ln}F+4F^2\left|{\rm
grad\,ln}F\right|^2,\\\label{LO1} &&\nabla^{ \varphi}_{{\rm
grad}\,F^2}\,{\rm d}{\varphi}({\rm grad\,ln}F)=2F^2\nabla^{
\varphi}_{{\rm grad\,ln}\,F}\,{\rm d}{\varphi}({\rm
grad\,ln}F),\\\label{LO2} && \nabla^{ \varphi}_{{\rm
grad\,ln}F}(F^2{\rm d}{\varphi}({\rm grad\,ln}F)) =2F^2\left| {\rm
grad\,ln}F\right|^2{\rm d}{\varphi}({\rm grad\,ln}F)\\\notag
&&+F^2\nabla^{
\varphi}_{{\rm grad\,ln}F}{\rm d}{\varphi}({\rm grad\,ln}F), \\
&&\nabla^{ \varphi}_{{\rm grad}\,F^2}\,\tau (\varphi,
g)=2F^2\nabla^{ \varphi}_{{\rm grad\,ln}\,F}\,\tau (\varphi, g),
{\rm and}\\\label{J200} && \nabla^{\varphi}_{{\rm
grad\,ln}F}(F^2\tau(\varphi, g))= 2F^2 \left|{\rm grad\,ln}F
\right|^2\tau(\varphi, g) +F^2\nabla^{ \varphi}_{{\rm
grad\,ln}F}\tau(\varphi, g).
\end{eqnarray}
Substituting $X= \tau(\varphi,{\bar g})= F^2\{\tau (\varphi,
g)-(m-2){\rm d}{\varphi}({\rm grad\,ln}F)\}$ into (\ref{JF}) and
using Equations (\ref{J100})$-$(\ref{J200}) we have
\begin{eqnarray}\notag
&&\tau^{2}(\varphi,{\bar g})=-J^{\varphi}_{\bar
g}(\tau(\varphi,{\bar g}))\\\notag &=&-F^2\{J^{
\varphi}_{g}(F^2\tau(\varphi, g)- (m-2)F^2{\rm d}{\varphi}({\rm
grad\,ln}F))\\\notag &&+(m-2)\nabla^{ \varphi}_{{\rm
grad\,ln}F}(F^2\tau(\varphi, g)- (m-2)F^2{\rm d}{\varphi}({\rm
grad\,ln}F))\}\\\notag &=&-F^2\{F^2J^{ \varphi}_{g}(\tau(\varphi,
g))-(\Delta F^2)\tau(\varphi, g))-2\nabla^{ \varphi}_{{\rm
grad}\,F^2}\,\tau(\varphi, g)\\\notag &&-(m-2)F^2J^{
\varphi}_{g}({\rm d}{\varphi}({\rm grad\,ln}F))+(m-2)(\Delta
F^2){\rm d}{\varphi}({\rm grad\,ln}F)\\\notag &&+2(m-2)\nabla^{
\varphi}_{{\rm grad}\,F^2}\,{\rm d}{\varphi}({\rm
grad\,ln}F)\\\notag &&+(m-2)\nabla^{ \varphi}_{{\rm
grad\,ln}F}(F^2\tau(\varphi, g)- (m-2)F^2{\rm d}{\varphi}({\rm
grad\,ln}F))\}\\\notag &=& F^4\{\tau^{2}(\varphi, g)+(m-2)J^{
\varphi}_{g}({\rm d}{\varphi}({\rm grad\,ln}F))\\\notag && +
2(\Delta {\rm ln}F-(m-4)\left|{\rm grad\,ln}F\right|^2)\tau(\varphi,
g)-(m-6)\nabla^{ \varphi}_{{\rm grad\,ln}\,F}\,\tau(\varphi,
g)\\\notag&& -2(m-2)(\Delta {\rm ln}F-(m-4)\left|{\rm
grad\,ln}F\right|^2){\rm d}{\varphi}({\rm grad\,ln}F)\\\notag
&&+(m-2)(m-6)\nabla^{ \varphi}_{{\rm grad\,ln}\,F}\,{\rm
d}{\varphi}({\rm grad\,ln}F)\}.
\end{eqnarray}
This gives the second part of the theorem.
\end{proof}

\begin{cor}
Let $\varphi : (M^{2},g) \longrightarrow (N^{n},h)$ be a map and
${\bar g}=F^{-2}g$ be a conformal change of the metric $g$. Let
$\tau^{2}(\varphi, g)$ and  $\tau^{2}(\varphi,{\bar g})$ be the
bitension fields of $\varphi$ with respect to the metrics $g$ and
${\bar g}$ respectively. Then,
\begin{eqnarray}\label{tao10}
\tau^{2}(\varphi,{\bar g})&=& F^4\{\tau^{2}(\varphi, g) +2 (\Delta
{\rm ln}F+2\left|{\rm grad\,ln}F\right|^2)\tau(\varphi, g))\\\notag
&& +4\nabla^{ \varphi}_{{\rm grad}\,\ln F}\,\tau(\varphi, g)\}.
\end{eqnarray}
\end{cor}
\begin{proof}
Substituting $m=2$ into the Equation (\ref{tao2}) we get
\begin{eqnarray}
&&\tau^{2}(\varphi,{\bar g})=F^4\tau^{2}(\varphi, g) +F^2\{ (\Delta
F^2)\tau(\varphi, g))+2\nabla^{ \varphi}_{{\rm
grad}\,F^2}\,\tau(\varphi, g)\}\\\notag && =F^4\{\tau^{2}(\varphi,
g) +2 (\Delta {\rm ln}F+2\left|{\rm
grad\,ln}F\right|^2)\tau(\varphi, g))+4\nabla_{{\rm grad}\,\ln
F}\,\tau(\varphi, g)\}.
\end{eqnarray}
\end{proof}
\begin{cor}\label{bk}
Let $\varphi : (M^{m},g) \longrightarrow (N^{n},h)$ be a harmonic
map with $m\neq 2$, and let ${\bar g}=F^{-2}g$ be a conformal change
of the metric $g$. Then, the map $\varphi : (M^{m},{\bar g})
\longrightarrow (N^{n},h)$ is a biharmonic map if and only if,
\begin{eqnarray}\label{tao20}
&&J^{ \varphi}_{g}({\rm d}{\varphi}({\rm grad\,ln}F))+(m-6)\nabla^{
\varphi}_{{\rm grad\,ln}F}{\rm d}{\varphi}({\rm grad\,ln}F)\\\notag
&&-2\left(\Delta {\rm\,ln}F-(m-4)\left| {\rm
grad\,ln}F\right|^2\right){\rm d}{\varphi}({\rm grad\,ln}F)=0.
\end{eqnarray}
\end{cor}
\begin{proof}
The corollary is obtained by applying Theorem \ref{confC} with
$\tau(\varphi, g)=\tau^2(\varphi, g)=0$ and $m\neq 2$.
\end{proof}

\begin{rk}
Let $\gamma =-\ln F$, then Corollary \ref{bk} recovers Proposition
2.1 in \cite{BK} after taking into account that their convention for
Laplacian on functions is $\Delta f=-{\rm trace} \nabla {\rm d} f$
which is different from ours by a negative sign.
\end{rk}
\begin{example}\label{E1}
The conformal immersion from Euclidean space into the hyperbolic
space
\begin{equation}\label{cfi}
\varphi : (\r^3\times \r^{+},\bar{g}=\delta_{ij}) \longrightarrow
(H^5=\r^4\times \r^{+},h=y_5^{-2}\delta_{\alpha\beta})
\end{equation}
 given by
$\varphi(x_1,\ldots,x_4)=(1,x_1,\ldots,x_4)$ is a proper biharmonic
map. In fact, the associated isometric immersion
\begin{equation}
\varphi : (\r^3\times \r^{+},g=x_4^{-2}\delta_{ij}) \longrightarrow
(H^5=\r^4\times \r^{+},h=y_5^{-2}\delta_{\alpha\beta})
\end{equation}
is totally geodesic and hence harmonic. Here, $\bar{g}=F^{-2}g$ with
$F=x_4^{-1}$. By Corollary \ref{bk}, the conformal immersion
(\ref{cfi}) is biharmonic if and only if Equation (\ref{tao20})
holds, which is equivalent to $J^{ \varphi}_{\bar{g}}({\rm
d}{\varphi}({\rm grad_{\bar{g}}\,ln}F))=0$. A straightforward
computation yields
\begin{eqnarray}\notag
J^{ \varphi}_{\bar{g}}({\rm d}{\varphi}({\rm
grad_{\bar{g}}\,ln}F))=-x_4^{-1}J^{ \varphi}_{\bar{g}}({\rm
d}{\varphi}(\partial_4))+\Delta (x_4^{-1}){\rm
d}{\varphi}(\partial_4)+2\nabla^{ \varphi}_{{\rm
grad}\,(x_4^{-1})}{\rm d}{\varphi}(\partial_4),
\end{eqnarray}
which is identically zero as one can check that
\begin{equation}\notag
-x_4^{-1}J^{ \varphi}_{\bar{g}}({\rm
d}{\varphi}(\partial_4))=-4x_4^{-3}\partial y^5,\;\;\Delta
(x_4^{-1}){\rm d}{\varphi}(\partial_4)=2x_4^{-3}\partial
y^5=2\nabla^{ \varphi}_{{\rm grad}\,(x_4^{-1})}{\rm
d}{\varphi}(\partial_4).
\end{equation}
\end{example}
\begin{example}
The conformal immersion from Euclidean space into the sphere
\begin{equation}\notag
\varphi : (\r^4,\bar{g}=\delta_{ij}) \longrightarrow (S^5\setminus
\{N\}\equiv \r^5,h=\frac{4\delta_{\alpha\beta}}{(1+|y|^2)^2})
\end{equation}
given by $\varphi(u_1,\ldots,u_4)=(u_1,\ldots,u_4,0)$, where
$(u_1,\ldots,u_5)$ are conformal coordinates on $S^5\setminus
\{N\}\equiv \r^5$, is a proper biharmonic map. In fact, the map is
the inverse stereographic projection that maps $\r^4$ into a great
hypersphere in $S^5$. The associated isometric immersion
\begin{equation}\notag
\varphi : (S^4\setminus
\{P\}\equiv\r^4,\bar{g}=\frac{4\delta_{ij}}{(1+|u|^2)^2})
\longrightarrow (S^5\setminus \{N\}\equiv
\r^5,h=\frac{4\delta_{\alpha\beta}}{(1+|y|^2)^2})
\end{equation}
is totally geodesic and hence harmonic. A computation similar to
those in Example \ref{E1} shows that the conformal immersion is
indeed a proper biharmonic map.
\end{example}
\section {Conformal biharmonic immersions}
\begin{prop}
Let $\varphi : (M^{m},g) \longrightarrow (N^{n},h)$ be a conformal
immersion with $\varphi^{*}h=\lambda^{2}g$. Let $\varphi :
(M^{m},{\bar g}) \longrightarrow (N^{n},h)$ be the associated
isometric immersion with mean curvature vector $\eta$, where ${\bar
g}=\varphi^{*}h=\lambda^{2}g$. Then, the conformal immersion
$\varphi : (M^{m},g) \longrightarrow (N^{n},h)$ is biharmonic if and
only if
\begin{eqnarray}\notag
\lambda^{4}\tau^{2}(\varphi,{\bar g})&=& -(m-2)J^{ \varphi}_{g}({\rm
d}{\varphi}({\rm grad\,ln}\lambda)) + 2m\lambda^2(-\Delta {\rm
ln}\lambda-2\left|{\rm grad\,ln}\lambda\right|^2)\eta\\\label{Confi}
&&+m(m-6)\lambda^2\nabla^{ \varphi}_{{\rm grad\,ln}\,\lambda}\,
\eta.
\end{eqnarray}
\end{prop}
\begin{proof}
Substituting $F=\lambda^{-1}$ and $\ln F=-\ln \lambda$ into the
Equation (\ref{tao2}) we have
\begin{eqnarray}\label{gd10}
&&\tau^{2}(\varphi,{\bar g})= \lambda^{-4}\{\tau^{2}(\varphi,
g)-(m-2)J^{ \varphi}_{g}({\rm d}{\varphi}({\rm
grad\,ln}\lambda))\\\notag && + 2(-\Delta {\rm
ln}\lambda-(m-4)\left|{\rm grad\,ln}\lambda\right|^2)\tau(\varphi,
g)+(m-6)\nabla^{ \varphi}_{{\rm grad\,ln}\,\lambda}\,\tau(\varphi,
g)\\\notag&& +2(m-2)(-\Delta {\rm ln}\lambda-(m-4)\left|{\rm
grad\,ln}\lambda\right|^2){\rm d}{\varphi}({\rm
grad\,ln}\lambda)\\\notag &&+(m-2)(m-6)\nabla^{ \varphi}_{{\rm
grad\,ln}\,\lambda}\,{\rm d}{\varphi}({\rm grad\,ln}\lambda)\}.
\end{eqnarray}
Note that the tension field of the conformal immersion $\varphi$ is
given by
\begin{eqnarray}\label{TCI}
\tau(\varphi)=m\lambda^2 \eta+(2-m){\rm d}\varphi \left( {\rm
grad}\, {\rm ln} \lambda\right).
\end{eqnarray}
Substituting (\ref{TCI}) into (\ref{gd10}) we have
\begin{eqnarray}
&&\tau^{2}(\varphi,{\bar g})= \lambda^{-4}\{\tau^{2}(\varphi,
g)-(m-2)J^{ \varphi}_{g}({\rm d}{\varphi}({\rm
grad\,ln}\lambda))\\\notag && + 2m\lambda^2(-\Delta {\rm
ln}\lambda-2\left|{\rm grad\,ln}\lambda\right|^2)\eta\\\notag
&&+m(m-6)\lambda^2\nabla^{ \varphi}_{{\rm grad\,ln}\,\lambda}\,
\eta\}.
\end{eqnarray}
 From this we obtain the proposition.
\end{proof}

\begin{thm}
The conformal immersion $\varphi : (M^{2},g) \longrightarrow
(\r^3,\langle,\rangle_{0})$ into Euclidean space with
$\varphi^{*}\langle,\rangle_{0}=\lambda^{2}g$ is biharmonic if and
only if
\begin{equation}\label{R3}
\begin{cases}
A_{\xi}({\rm grad} H)+ \frac{1}{2}{\rm grad} (H^2)+2H\,A_{\xi}({\rm
grad\;ln} \lambda)=0\\\Delta H -H\,|B|^2+2H(\Delta {\rm
ln}\lambda+2\left|{\rm grad\,ln}\lambda\right|^2)+4g({\rm grad\;ln}
\lambda,{\rm grad} H)=0,
\end{cases}
\end{equation}
where $\xi$ is the unit normal vector field of the surface
$\varphi(M)\subset \mathbb{R}^3$ and $A_{\xi}$ and $H$ are the shape
operator and the mean curvature function of the surface
respectively.
\end{thm}
\begin{proof}
It follows from (\ref{Confi}) with $m=2$ that the conformal
immersion $\varphi$ is biharmonic if and only if
\begin{eqnarray}\label{GD12}
&&\lambda^{2}\tau^{2}(\varphi,{\bar g})=  -4(\Delta {\rm
ln}\lambda+2\left|{\rm grad\,ln}\lambda\right|^2)\eta-8\nabla^{
\varphi}_{{\rm grad\,ln}\,\lambda}\,\eta,
\end{eqnarray}
where $\tau^{2}(\varphi,{\bar g})$ denotes the bitension field of
the associated isometric immersion $\varphi :
(M^{2},\bar{g}=\lambda^{2}g) \longrightarrow \mathbb{R}^3$ with mean
curvature vector $\eta=H\xi$, where $\xi$ and $H$ are the unit
normal vector field and the mean curvature function of the surface
$\varphi (M)$ respectively. Then, we have (see, e.g., \cite{Ji2},
\cite{CH} and \cite{CMO2})
\begin{eqnarray}\notag
\tau^{2}(\varphi, \bar{g}) = 2(\Delta_{\bar{g}} H
-H\,|B|_{\bar{g}}^2 )\xi - 2[ 2A_{\xi}({\rm grad}_{\bar{g}} H)+ {\rm
grad}_{\bar{g}} (H^2)],
\end{eqnarray}
substitute this into (\ref{GD12}) we have
\begin{eqnarray}\label{gd20}
&&\lambda^{2}(\Delta_{\bar{g}} H -H\,|B|^2 )\xi - \lambda^{2}[
2A_{\xi}({\rm grad}_{\bar{g}} H)+ {\rm grad}_{\bar{g}}
(H^2)]\\\notag &=& -2(\Delta {\rm ln}\lambda+2\left|{\rm
grad\,ln}\lambda\right|^2)H\xi-4\nabla^{ \varphi}_{{\rm
grad\,ln}\,\lambda}\,H\xi.
\end{eqnarray}
Notice that the transformations of Laplacian and the gradient
operators under a conformal change of metrics $\bar{g}=\lambda^{2}g$
in two dimensional manifold are given by
\begin{eqnarray}\label{La}
\Delta_{\bar{g}} u=\lambda^{-2}\Delta u,\;\;\;{\rm
grad}_{\bar{g}}u=\lambda^{-2}{\rm grad}u.
\end{eqnarray}
On the other hand, we have
\begin{eqnarray}\label{La1}
-4\nabla^{ \varphi}_{{\rm grad\,ln}\,\lambda}\,H\xi=-4g({\rm
grad\,ln}\lambda,{\rm grad}H) \xi +4H\,A_{\xi}({\rm
grad\;ln}\lambda).
\end{eqnarray}
Using (\ref{La}) and substituting (\ref{La1}) into (\ref{gd20}) and
comparing the tangential and the normal components  we obtain
equation (\ref{R3}) which completes the proof of the theorem.
\end{proof}
\begin{prop}\label{EZHU}
For $\lambda^2=\big(C_2e^{\pm z/R}-C_1C_2^{-1}R^2e^{\mp z/R}\big)/2$
with constants $C_1, C_2$, the maps $\phi:( D, g=\lambda^{-2}(R^2
d\theta^2+dz^2))\longrightarrow
(\r^3,d\sigma^2=d\rho^2+\rho^2\,d\theta^2+dz^2)$ with $\phi(\theta,
z)=(R, \theta, z)$ is a family of proper biharmonic conformal
immersions of  a cylinder of radius $R$ into Euclidean space $\r^3$.
\end{prop}
\begin{proof}
Let $\phi:\r^2\supseteq D\longrightarrow \r^3$,
$\phi(\theta,z)=(R\cos\,\theta, R\sin\,\theta, z)$ be the isometric
immersion with the image $\phi(D)$ being a cylinder of radius $R$ in
$3$-space. Using cylindrical coordinates $(\rho, \theta, z)$ on
$\mathbb{R}^3$ we can represent the isometric immersion of the
cylinder as $\phi:\r^2\supseteq D\longrightarrow \r^3$ with
$\phi(\theta, z)=(R, \theta, z)$. It is easy to check that$
E_1=\frac{\partial}{\partial \rho},\;\;
E_2=\frac{1}{\rho}\frac{\partial}{\partial
\theta},\;\;E_3=\frac{\partial}{\partial z}$ constitute a local
orthonormal frame of $\r^3$ and that $e_1=E_2,\;\; e_2=E_3, \xi=E_1$
is an adapted orthonormal frame along the cylinder with $\xi$ being
unit normal vector field. We can check that the induced metric on
the cylinder is ${\bar g}=R^2d\theta^2+dz^2$. Let
$g=\lambda^{-2}(R^2d\theta^2+dz^2)$ be a conformal change of the
metric on the cylinder. Then, we have a conformal immersion $\phi:(
D, g)\longrightarrow (\r^3,
d\sigma^2=d\rho^2+\rho^2\,d\theta^2+dz^2)$ with
$\phi^{*}d\sigma^2=\lambda^2g={\bar g}$. A straightforward
computation gives
\begin{eqnarray}\label{gd30}
\begin{cases}
A_{\xi}e_1=-\frac{1}{R}e_1,\;\;A_{\xi}e_2=0,\\
H=\frac{1}{2}(\langle A_{\xi}e_1,e_1\rangle+\langle
A_{\xi}e_2,e_2\rangle)=-\frac{1}{2R}\ne 0\\
|B|^2=\lambda^2|B|^2_{{\bar g}}=\lambda^2\sum_{i=1}^2|A(e_i)|^2=\lambda^2\frac{1}{R^2},\\
{\rm grad}\,H=0,\\
 \Delta H=0.
 \end{cases}
\end{eqnarray}
Substituting (\ref{gd30}) into (\ref{R3}) we conclude that conformal
immersion $\phi$ is biharmonic if and only if
\begin{equation}\notag
\begin{cases}
A_{\xi}({\rm grad\;ln} \lambda)=0\\
\lambda^2-2R^2(\Delta {\rm ln}\lambda+2\left|{\rm
grad\,ln}\lambda\right|^2)=0.
\end{cases}
\end{equation}
It is not difficult to check that this system is equivalent to
\begin{equation}\notag
\begin{cases}
\lambda(\theta, z)=\lambda(z)\\
1-2R^2[({\rm ln}\lambda)''+2({\rm ln}\lambda)'^2)]=0
\end{cases}
\end{equation}
or,
\begin{equation}\notag
(\lambda^2)''=\frac{1}{R^2}\lambda^2.
\end{equation}
It follows that $\lambda^2$ is a solution of the ordinary
differential equation
\begin{equation}\notag
y''=\frac{1}{R^2}y,
\end{equation}
which has (see e.g., \cite{Cu}) the first integral
\begin{equation}\label{GD10}
y'^2=y^2/R^2+C_1.
\end{equation}
Solving Equation (\ref{GD10}) we have
\begin{equation}
y=\big(C_2e^{\pm z/R}-C_1C_2^{-1}R^2e^{\mp z/R}\big)/2.
\end{equation}
Notice that the conformal immersion has nonzero constant mean
curvature $H$ so it is not harmonic. Therefore, we complete the
proof of the proposition.
\end{proof}
\begin{remark}
It follows from Proposition \ref{EZHU} that the biharmonic conformal
immersions of the cylinder in $\r^3$ are not minimal, thus the
well-known fact that a conformal harmonic immersion of a surface
must be a minimal surface fails to generalize to conformal
biharmonic immersion of a surface. Our proposition also shows that
if B. Y. Chen's conjecture \cite{CH} about biharmonic isometric
immersions into Euclidean space is generalized to biharmonic
conformal immersions, then the answer is negative.
\end{remark}

\section{Biharmonicity and holomorphicity of conformal immersions}

Let $\varphi : (M^2 ,g) \longrightarrow \r^n $ be a conformal
immersion of a Riemann surface. Let $(u, v)$ be the local
coordinates on $M$ and we write $z=u+iv$ in the local complex
parameter. We also use the usual notations
\begin{equation}\notag
\frac{\partial}{\partial z}=\frac{1}{2}(\frac{\partial}{\partial
u}-i \frac{\partial}{\partial v}),\;\; {\rm and}\;\;
\frac{\partial}{\partial {\bar
z}}=\frac{1}{2}(\frac{\partial}{\partial u}+i
\frac{\partial}{\partial v}).
\end{equation}
Then, the well-known Weierstrass representation theorem for
conformal harmonic immersions can be stated as: Let $\varphi : (M^2
,g) \longrightarrow (\r^n, \langle ,\rangle_{0})$ be a harmonic
conformal immersion. Then, the section $ \phi=\frac{\partial
\varphi}{\partial z}=\frac{1}{2}( \varphi_{u}-i
\varphi_{v})=\phi^\alpha(z)\frac{\partial}{\partial y^\alpha}$ is
holomorphic and satisfies
\begin{eqnarray}\label{W1}
\sum _{\alpha=1}^{n}(\phi^\alpha)^2=0,\\\label{W2} \sum
_{\alpha=1}^{n}|\phi^\alpha|^2\ne 0.
\end{eqnarray}
Conversely, given any  holomorphic section $
\phi=\phi^\alpha\frac{\partial}{\partial y^\alpha}:M\longrightarrow
\mathbb{E}$ satisfying (\ref{W1}) and (\ref{W2}) and the periodic
condition:
\begin{equation}\notag
\mathfrak{Re}\int_{\gamma}(\phi^1,\ldots, \phi^n){\rm d}z=0,
\end{equation}
 for any closed path in $M$. Then, the map
 \begin{equation}\notag
 \varphi(z)=2\,\mathfrak{Re}\int_{z_0}^{z}(\phi^1,\ldots, \phi^n){\rm d}z
 \end{equation}
 defines a harmonic conformal immersion of a Riemann surface into
 the Euclidean space.\\

For conformal biharmonic immersions of surfaces into Euclidean
space, we have
\begin{thm}\label{Weie}
$\varphi : (M^2 ,g) \longrightarrow (\r^n, \langle ,\rangle_{0})$ is
a conformal biharmonic immersion with
$\varphi^{*}\langle,\rangle=\lambda^2g$ if and only if the section $
\phi=\frac{\partial \varphi}{\partial z}=\frac{1}{2}( \varphi_{u}-i
\varphi_{v})=\phi^\alpha(z)\frac{\partial}{\partial y^\alpha}$
satisfies Equations (\ref{W1}), (\ref{W2}) and
\begin{eqnarray}\label{W3}
\frac{\partial}{\partial {\bar z}}\frac{\partial}{\partial
z}\left(\lambda^{-2}\frac{\partial \phi^\sigma}{\partial {\bar
z}}\right)=0
\end{eqnarray}
\end{thm}
\begin{proof}
Using the local conformal parameter $z=u+iv$ and the characteristic
property that $g=\varphi^{*}\langle ,\rangle_{0}$ of the conformal
immersion $\varphi : (M^2 ,g) \longrightarrow (\r^n, \langle
,\rangle_{0}) $ we can be write the metric $g$ as $g=\lambda^2({\rm
d}u^2+{\rm d}v^2)=\lambda^2|{\rm d} z|^2$, where $\lambda^2=\langle
\varphi_{u} ,\varphi_{u}\rangle_{0}=\langle \varphi_{v}
,\varphi_{v}\rangle_{0}$. The Laplacian operator on $(M, g)$ can be
written as
\begin{equation}\notag
\Delta=\lambda^{-2}(\frac{\partial^2}{\partial
u^2}+\frac{\partial^2}{\partial
v^2})=4\lambda^{-2}\frac{\partial}{\partial {\bar
z}}\frac{\partial}{\partial z}.
\end{equation}

Let $\{y^\alpha, \frac{\partial}{\partial y^\alpha}\}$ be local
coordinates in a neighborhood $U$ of $\r^n$ such that $U\cap \varphi
(M)$ is nonempty. We can write the local expression
$\varphi(z)=(\varphi^1(z),\ldots,\varphi^n(z))$ as
$\varphi(z)=\varphi^\alpha(z)\frac{\partial}{\partial y^\alpha}$. If
we define the section $ \phi=\frac{\partial \varphi}{\partial
z}=\frac{1}{2}( \varphi_{u}-i \varphi_{v})$. Then, it is well-known
(see, e.g., \cite{ES}) the tension field of $\varphi$ can be written
as
\begin{equation}\notag
\tau (\varphi)=(\Delta\varphi^1,\ldots,\Delta\varphi^n)=
4\lambda^{-2}\frac{\partial^2 \varphi^\sigma}{\partial {\bar
z}\partial z}\frac{\partial}{\partial
y^\sigma}=4\lambda^{-2}\frac{\partial \phi^\sigma}{\partial {\bar
z}}\frac{\partial}{\partial y^\sigma},
\end{equation}
and hence the bitension field is given by
\begin{eqnarray}\label{bih49}
\tau^2 (\varphi)=(\Delta^2\varphi^1,\ldots,\Delta^2\varphi^n)
=4\lambda^{-2}\frac{\partial}{\partial {\bar
z}}\frac{\partial}{\partial z}\left(4\lambda^{-2}\frac{\partial
\phi^\sigma}{\partial {\bar z}}\right)\frac{\partial}{\partial
y^\sigma}.
\end{eqnarray}
It is easy to see that Equations (\ref{W1}) and (\ref{W2}) is
equivalent to $\varphi$ being a conformal immersion whilst Equation
(\ref{W3}) is equivalent to the biharmonicity of $\varphi$ by
(\ref{bih49}).
\end{proof}

\begin{example}
We can use Weierstrass representation to prove that the map
$\varphi:(\r^2,g=e^{y/R}(dx^2+dy^2)\longrightarrow \r^3$,
$\varphi(x,y)=(R\cos\,\frac{x}{R}, R\sin\,\frac{x}{R}, y)$ is a
proper biharmonic conformal immersion of  $\r^2$ into Euclidean
space $\r^3$.\\

Indeed, in this case, $\varphi_x=(-\sin \frac{x}{R}, \cos
\frac{x}{R}, 0),\;\;\varphi_y=(0, 0, 1)$ and $\varphi$ is a
conformal immersion with $\varphi^{*}\langle,\rangle_{0}={\bar
g}=dx^2+dy^2=\lambda^2g$ for $\lambda^2=e^{-y/R}$. The section $
\phi=\frac{\partial \varphi}{\partial z}=\frac{1}{2}( \varphi_{x}-i
\varphi_{y})=\frac{1}{2}(-\sin \frac{x}{R}, \cos \frac{x}{R}, -i)$
with components
\begin{equation}
\phi^1=-\frac{1}{2}\sin\,\frac{z+{\bar
z}}{2R},\;\;\phi^2=\frac{1}{2}\cos\,\frac{x}{R}=\frac{1}{2}\cos\,\frac{z+{\bar
z}}{2R},\;\;\phi^3=-i/2.
\end{equation}
A straightforward computation yields
\begin{equation}
\begin{cases}
\lambda^{-2}\frac{\partial \phi^1}{\partial
\bar{z}}=-\frac{1}{4R}e^{-(z-{\bar z})i/(2R)}\cos\,\frac{z+{\bar
z}}{2R} =-\frac{i}{8R}(e^{\bar{z}i/R}+e^{-zi/R}),\\
\lambda^{-2}\frac{\partial \phi^2}{\partial
\bar{z}}=-\frac{1}{4R}e^{-(z-{\bar z})i/(2R)}\sin\,\frac{z+{\bar
z}}{2R} =-\frac{i}{8R}(e^{e^{-zi/R}-\bar{z}i/R}),\\
\lambda^{-2}\frac{\partial \phi^3}{\partial \bar{z}}=0.
\end{cases}
\end{equation}
Clearly, we have $\frac{\partial}{\partial {\bar
z}}\frac{\partial}{\partial z}\big(\lambda^{-2}\frac{\partial
\phi^\sigma}{\partial {\bar z}}\big)=0$ for $\sigma=1, 2, 3$ and
hence, by Theorem \ref{Weie}, $\varphi$ is a biharmonic conformal
immersion which is not harmonic as the section $\phi$ is not
holomorphic.
\end{example}
\begin{example}
We can also easily check that the map
$\varphi:(\r^2,g=e^{y/R}(dx^2+dy^2)\longrightarrow \r^6$,
$\varphi(x,y)=(R\cos\,\frac{x}{R}, R\sin\,\frac{x}{R},
y,R\cos\,\frac{x}{R}, R\sin\,\frac{x}{R}, y)$ is a proper biharmonic
conformal immersion of  $\r^2$ into Euclidean space $\r^6$.
\end{example}

\end{document}